\documentclass{amsart}     
\usepackage{amscd,amssymb,hyperref}
\usepackage{pstricks}
\usepackage{amsthm,latexsym,eucal}
\title{Algebra, Topology and Algebraic Topology of 3D Ideal Fluids}
\author{Dennis Sullivan}
\theoremstyle{plain}
\newtheorem{prop}[subsection]{Proposition}

\newtheorem{cor}[subsection]{Corollary}
\newtheorem{conc}[subsection]{Conclusion}
\theoremstyle{remark}

\theoremstyle{definition}
\newtheorem{defn}[subsection]{Definition}
\newtheorem{exm}[subsection]{Example}

\numberwithin{equation}{section}

\newcommand{\Z}{{\mathbb Z}}

\begin{document}

\begin{abstract}  There is a remarkable and canonical problem in 3D geometry and topology: To understand existing  models of 3D fluid motion
or to create new ones that may be useful. We discuss from an algebraic viewpoint the PDE called Euler's equation for incompressible frictionless fluid motion.  
In part I we define a ``finite dimensional 3D fluid algebra,'' write its Euler equation and derive properties related to energy, helicity,  transport of vorticity and linking that characterize this equation. This is directly motivated by the infinite dimensional fluid algebra associated to a closed riemannian three manifold whose Euler equation as defined above is the  Euler PDE of fluid motion.
 The classical infinite dimensional fluid algebra satisfies an additional identity related to the Jacobi identity for the lie bracket of vector fields.
In part II we discuss informally how this Jacobi identity can be reestablished in finite dimensional approximations as a Lie infinity algebra. The main point of a developed version of this theory would be a coherence between various levels of approximation. It is hoped that a better understanding of the meaning of the Euler equation in terms of such infinity structures would yield algorithms of computation that work well for conceptual reasons.
\end{abstract}
\maketitle
\section{Algebra and Topology of Ideal Fluids}

A finite dimensional $3D$-fluid algebra is a finite dimensional vector space $V$ provided with three structures:\begin{itemize}
\item[1)] an alternating  trilinear form $\{\, ,\, ,\, \}$ on $V$, called the triple intersection form.
\item[2)] a symmetric nondegenerate bilinear form $\langle\, ,\, \rangle$ on $V$, called the vorticity linking form.
\item[3)] a positive definite inner product $(\, ,\, )$ on $V$, called the metric.
\end{itemize}

If $M$ is a $3D$  closed oriented Riemannian manifold there is a classical example of a fluid algebra which is infinite dimensional and which is constructed inside the differential forms:\begin{itemize}
\item  $V$ consists of the coexact one forms, the image of two forms by the operator $*d*$ where $*$ is the Hodge star operator of the metric. Under the correspondence between one forms and vector fields given by the metric, elements in $V$ correspond to volume preserving vector fields which have flux or net flow  zero across any closed hypersurface.
\item The vorticity linking form on $V$ is defined by 
setting \[ \langle a,b\rangle=\int_M a\wedge db.\]  This is equal to $\int_M
(da)\wedge b$, depends only on $da$ and $db$, and
may be construed as a linking number
of the two one-dimensional transversally measured foliations defined by the kernels of $da$ and $db$ and transversally measured by the two forms $da$ and $db$ (see Arnold-Khesin\cite{AK} and Sullivan\cite{S}). Here $da$ and $db$ determine the vorticities of the vector fields corresponding to $a$ and to $b$. $da$ and $db$ may be approximated, in the sense of integrating against a smooth test form, by weighted sums of closed curves which  bound weighted sums of surfaces approximating $a$ and $b$. Recall that linking numbers are defined by intersecting one  set of curves with  surfaces bounding a second set of curves. Thus the integral approximately  computes the total linking of the weighted family of curves approximating $da$ with those approximating $db$.  

One can see from the Hodge decomposition that the vorticity linking form is nondegenerate on coexact one-forms.  
\item  The triple intersection form $\{a,b,c\}$ is the integral of $a$ wedge $b$ wedge $c$ over $M$, which may be construed as a triple intersection of the surfaces with boundary approximating $a$, $b$, and $c$.\end{itemize}

The Euler evolution in $V$ describing ideal incompressible frictionless fluid motion in $M$ is an $ODE$ whose solutions may be described in words as follows: an isotopy from the identity is a fluid motion iff the path in $V$ corresponding to the instantaneous velocity of the motion satisfies: the vorticity at time $t$ is the two form transported by the motion from the vorticity two form at time zero.

We can write out this Euler $ODE$ just using the elements of a fluid algebra.

 Namely, the right hand side of the evolution $ODE$ for $X(t)$ is described by its inner product with any vector $Z$ as follows:  $\left(\frac{dX}{dt},Z\right)$ = $\{X,DX,Z\}$ where $D$ is the operator on $V$ defined by $(DX,Y)$ = $\langle X, Y\rangle$.
$D$ is called the curl operator on the  elements of a fluid algebra. Note that $(DX,Y)$ = $(X,DY)$ since $\langle \, ,\, \rangle$ is symmetric by assumption.

\begin{prop} The ``energy'' =  $(X,X)$ and  the ``helicity'' =  $(X,DX)$ are each preserved by the evolution in a  fluid algebra.\end{prop}
\begin{proof} \begin{itemize} \item[1)] $\frac{d}{dt}\mbox{(energy)}=2\left(\frac{dX}{dt},X\right) = 2\{X,DX,X\} =0$ by the alternating property of $\{\,  , \, , \,\}$.
\item[2)] $\frac{d}{dt}\mbox{(helicity)}=2\left(\frac{dX}{dt},DX\right)=2\{X,DX,DX\}=0$  again by the alternating property of $\{\, ,\, ,\, \}$.\end{itemize}
\end{proof}
\begin{cor}For a finite dimensional fluid algebra, there is no finite time blowup and the flow stays on the intersection of the energy sphere with the helicity level set.
\end{cor}

  Notice in this classical three manifold case the alternating trilinear form on all one-forms and the  bilinear form on all one-forms  are purely topological and don't depend on the metric. The metric is needed to define the subspace of coexact forms inside all one-forms  to which one restricts the trilinear form, the bilinear form and the inner product to obtain the infinite dimensional fluid algebra.

Also notice that even though the Proposition follows for  this classical Euler case by the same short proof, the Corollary of course does not follow because of the non compactness of the infinite dimensional sphere.
The truth or falsity of infinite time existence of smooth solutions of the Euler evolution for arbitrary smooth initial conditions is a celebrated open problem.

Here is another family of examples:
\begin{exm} Given a finite dimensional Lie algebra $[\, ,\, ]$ with a nondegenerate invariant symmetric pairing $\langle\, ,\, \rangle$ define a fluid algebra by:\begin{itemize}
\item for  the linking form take $\langle\, ,\, \rangle$.
\item  for  the intersection form take $\{X,Y,Z\}$ =$\langle [X,Y] , Z \rangle$ and 
\item for  the metric take any positive definite inner product.\end{itemize}
Euler flows for these examples were described differently by Arnold and
Khesin \cite{AK}.  The interested reader may verify that their evolution equation (see \cite[Theorem I.4.9]{AK}) is related to ours by a linear coordinate change.\end{exm} Notice:\begin{itemize}
\item[1)]For a general fluid algebra one may reverse the above formula to define a bracket by the formula $\langle [X,Y] , Z \rangle$= $\{X,Y,Z\}$. If the Jacobi identity were satisfied for this bracket the fluid algebra would arise from the Lie algebra example.
\item[2)] In the infinite dimensional example related to $3D$ manifolds and the Euler ODE this Jacobi identity is satisfied and corresponds to the Lie  algebra of  volume preserving flux zero vector fields.
\item[3)] In order to have models which satisfy the invariance of energy, helicity, and more generally circulation (see below) and the transport property of vorticity (see below) this Jacobi relation is not required. The significance of the Jacobi relation
 needs exploring (see Part II). One interpretation of Jacobi is the following: if we extend the transport (defined below) to higher tensors by the derivation property then Jacobi is equivalent to the statement that transport fixes the tensor defining transport itself.
\end{itemize}
\begin{defn}{(Velocity and Vorticity)} We think of elements in $V$ as analogous to velocity and vorticity fields of the ideal fluid. The  vorticity of a velocity  is obtained by applying the curl operator $D$ to the velocity.\end{defn}

\begin{prop} If $X(t)$ satisfies the Euler equation of the fluid algebra \[ \left(\frac{dX(t)}{dt},Y\right)=  \{X,DX,Y\},\] then the vorticity $DX(t)=Y(t)$ satisfies \[ \left(\frac{dY(t)}{dt},Z\right)=\{D'Y,Y,DZ\} \]  where $D'$ is the inverse of $D$.\end{prop}

\begin{proof} $\frac{d}{dt}DX(t)=D\left(\frac{d}{dt}X(t)\right)$. So \begin{align*}\left(\frac{d}{dt}Y(t),Z\right)&= \left(D\left(\frac{d}{dt}X(t)\right),Z\right)= \left(\frac{d}{dt}X(t),DZ\right)\\&=\{X,DX,DZ\} = \{X,Y,DZ\} = \{D'Y,Y,DZ\}.\end{align*}\end{proof}

\begin{defn}{(Transport)} Define the infinitesimal transport of $Z$ by $X$, $T(X,Z)$, by the condition that its inner product with arbitrary $W$ is given by
 $(T(X,Z),W)=\{X,Z,DW\}$.\end{defn}  Note $T(X,Z)$ = $-T(Z,X)$, but Jacobi is not assured for $T(X,Y)$.
 The transport is meant to model for volume preserving vector fields the Lie derivative action, which by the way is canonically equivalent to the the Lie derivative action on closed two forms.

\begin{cor} If $X(t)$ satisfies  the Euler equation of a fluid algebra then the vorticity $Y(t) =\mbox{curl}X(t) =DX(t)$ satisfies $\frac{dY(t)}{dt} = T(X(t),Y(t))$, namely the vorticity is transported from one time to another by the motion.\end{cor}

\begin{proof} Since $Y=DX$, by the proposition and the definition of transport the inner product of each with $Z$ equals $\{X,DX,DZ\}$.\end{proof}

Note: This definition of transport  may be derived in the classical case from the expression of the  Lie bracket of  two volume preserving vector fields $[V,V']$ in terms of the corresponding one-forms by the formula  $[V,V']=*d*(V\wedge V')$ where $*d*$ is the adjoint of $d$.

\begin{cor}{(Invariance of circulation)}\label{invnce} If the vorticity  $DZ(t)$ of a time dependent field $Z(t)$ satisfies $\frac{d}{dt}(DZ(t)) = T(X(t),DZ(t))$, then the linking number of $DZ(t)$ and $DX(t)$ stays constant.\end{cor}

Here, consistently with the situation discussed in the classical example, the linking number of $DZ(t)$ and $DX(t)$ refers to the value $\langle X,Z\rangle$ of the vorticity linking form on the fields $X$ and $Z$.

\begin{proof} The linking number of $DZ$ and $DX$ is equal to $(X,DZ)$. So \begin{align*}\frac{d}{dt}(X,DZ)&= \left(\frac{d}{dt}X(t),DZ(t)\right) + \left(X,\frac{d}{dt}DZ(t)\right) \\&= \{X,DX,DZ\} + (T(X,DZ),X)
 \\& =\{X,DX,DZ\} + \{X,DZ,DX\}\end{align*}  which is zero by the alternating property.\end{proof}

\begin{conc} The Euler ODE for velocity associated to a fluid algebra is the only ODE whose evolution keeps constant the linking of the transported vorticity of a general field with the evolving vorticity of the velocity solution of the ODE (as expressed in Corollary \ref{invnce}).\end{conc}
\begin{proof}
The  calculation of the proof of Corollary \ref{invnce} shows cancellation takes place iff the evolution of the velocity $X(t)$ satisfies 
$\left(\frac{d}{dt}X(t),DZ\right) =\{X,DX,DZ\}$. Since $D$ is invertible this is equivalent to the definition of the Euler ODE.
\end{proof}

\section{Algebraic Topology of Ideal Fluids}

In the classical example of ideal fluid motion on a metric three manifold $M$ without boundary, the fluid algebra is embedded in the differential algebra of all differential forms. So we are adding to the fluid algebra non constant functions, non constant volume
 forms, more one-forms and more two forms. New operators and operations appear: exterior derivative, the wedge product, the integral of three forms over the manifold, the  Hodge star operator and the Hodge decomposition with the associated projections:
   all forms = exact forms + coexact forms + harmonic forms.  (exact means image $d$ and coexact means image of $*d*$).
The fluid algebra in these terms is made from the subspace of coexact forms in degree one, the curl operator obtained by doing $d$ then star, the alternating form and metric as described above using wedge, star and the integral.
Alas, all of this is infinite dimensional so compactness and the easy proof above of long time existence disappears.

However there are many finite dimensional models of the algebraic topology residing in the above structure.
Two types are:\begin{itemize}
\item[I] (grid type). Divide the manifold into cells and average the forms over oriented faces of the correct dimension. By Stokes' theorem
one gets a map of chain complexes $I$:(all smooth forms with $d$) $\rightarrow$ (cellular cochains with $d$) which induces an isomorphism on cohomology. There are maps of chain complexes in the opposite direction using the heat flow and dual cells which are inverse to $I$ up to chain homotopy and invariant under star (from class lectures and an unpublished manuscript).

\item[II] (eigenvalue type). Consider the eigendecomposition of all $p$ forms for the laplacian operator $(-1)^{p-1}(*d*d- d*d*)$ and project onto finite dimensional pieces by putting bounds on the eigenvalues. These are invariant under $d$ and star as well and induce isomorphisms on cohomology. For the flat torus these are used in numerical calculation for fluids.\end{itemize}

A third type which is due to Whitney, is very elegant and useful for algebraic topology (see \cite{S3}) but it is not invariant under star for a fundamental reason which to my knowledge has never been corrected. So we do not consider Whitney forms here.

The first two finite dimensional approximations have everything they need for defining the fluid algebra approximating the classical fluid algebra for the classical Euler evolution  EXCEPT the wedge product. This lack can be corrected by forming the wedge product on the finite dimensional image inside forms and then integrating or projecting back to the finite dimensional model.  This is how we multiply numbers on a computer using finite parts of the decimal expansion and how numerical computations for fluids are performed.
In each case the multiplication  obtained  is appropriately commutative  but NOT associative.

Now nontrivial ideas of algebraic topology enter. 
First recall  that a  chain mapping of chain complexes
of vector spaces (or free $\Z$-modules)  inducing a homology isomorphism 
has an inverse up to chain homotopy.  Such inverses and the chain homotopies can be used to transport algebraic structures (up to homotopy) between chain complexes that are  very different as vector spaces, for example the infinite dimensional deRham complex and the finite dimensional models.
The associator of this finite dimensional wedge product $((a \wedge b) \wedge c)$ - $(a \wedge (b \wedge c))$ is a three to one operation which is a mapping of chain complexes which commutes with the the natural differentials. Moreover because of the homology isomorphism above this associator is the commutator with $d$ of a correcting three to one operation of degree one less which we denote $\wedge_1$. Now we continue using ideas of Stasheff's Princeton Thesis (1959). Consider the five ways to associate four entities. These arrange naturally at the vertices of a pentagon. We  sum the corrections $\wedge_1$ on each  edge  combined with $\wedge$
to build a four to one operation of degree $-1$ which again is a cycle in that it commutes with $d$ extended to multivariables by the Leibniz rule. By the homology isomorphism above there is a degree $-2$ four to one operation $\wedge_2$ which fills in the cycle corresponding to the boundary of the pentagon, i.e. one whose commutator with $d$ is the four to one operation of degree $-1$ obtained by going around the pentagon. This process continues indefinitely producing $n$ to $1$  correction operations of degree $-n$ whose commutator with $d$ is an appropriate more and more complex looking formula in the inductively constructed corrections.

There are two versions of this: the original Stasheff one is very beautiful. The complex looking formulae are nothing more than the combinatorics of the moduli spaces of complete hyperbolic surfaces with geodesic boundary boundedly related to  the two disk with three or more punctures on the boundary, then naturally compactified by geometric limits. This is the model also controlling genus zero  part of open string theories. The Stasheff  polyhedra can also be described as a moduli space of planar rooted trees.

The second way remembers the commutativity and seems more appropriate here. It may be modeled on  the moduli spaces of
rooted trees (in space) whose leaves are labeled and whose interior edges are painted black or white. This model is a picture of what is called the bar cobar construction for algebras over an operad (see Vallette\cite{V} and Wilson\cite{W}).

Interestingly enough there is a specific procedure for computing these corrections by placing the chain homotopies mentioned above on the interior edges of the trees and using the wedge at the vertices. This procedure is identical to the tree part of the Feynman diagram algorithm in perturbative Chern-Simons quantum field theory where the ``propagator'' there is the chain homotopy here.

After doing this work we obtain the  derived or ``infinity''  version of the graded commutative wedge product compressed onto the finite dimensional approximating models.

The same ideas may be applied to the bracket or transport discussed above which satisfies Jacobi in the infinite dimensional model. The Jacobi identity may be encoded in the the algebraic statement that on forms the wedge product and the adjoint of $d$ = $(-1)^p *d*$ satisfy the following identity in words: the deviation of $*d*$ from being a derivation of $\wedge$ is, as a two variable operator,  itself a derivation in each variable. Adding the equation that $*d**d*$ = $0$ yields formally that the two variable operation satisfies Jacobi. This Lie bracket on forms becomes via the metric isomorphism the usual Lie bracket of vector fields extended by Leibniz to all multivector fields.  (This is referred to as the Schouten-Nijenhuis bracket.)
This relationship between $*d*$ and $\wedge$  means we have a  so called Batalin-Vilkovisky algebra or briefly a $BV$ algebra.
 The $BV$ formalism in perturbative quantum field theory due to Batalin-Vilkovisky is  perhaps the most natural for mathematicians (see Costello \cite{C}).
The $BV$ structure $(*d*,\wedge)$ may also be
compressed to a new structure of BV algebra up to homotopy into our finite dimensional models. The bracket so obtained is called a Lie infinity structure.

So now we have a commutative infinity structure related to $d$ and a  Lie infinity structure related to $*d*$   whose leading term is the transport above.

Before rushing off to make models that would be used for fluid simulation these ideas need I think to be completed in the following way:\begin{itemize}
\item[1)] describe the complete structure used above $d$,star, $\wedge$ and the integral as an algebraic structure. Then develop the diagrammatic compression algorithms for this structure which yield the derived or infinity version of the structure on the finite dimensional approximations. To my knowledge the current abstract homotopical algebra just falls short of this task.  General multilinear operations with outputs can be treated as in \cite{S2}, but pairings are not treated there or anywhere else to my knowledge (although Kevin Costello's work on renormalization and perturbative quantum field theory comes close \cite{C}).
\item[2)] understand the Euler evolution as a functorial construction on the derived or infinity version of this algebraic structure. The efforts of Arnold and Khesin and others  are a beginning but to my knowledge this goal is not yet achieved (although describing the flow by an action principle may lead to such a functorial principle).
\end{itemize}
If these two tasks are completed, we will have conceptually natural effective theories of fluid motion at every scale that fit together in an appropriate sense. There will be, by definition almost, natural algorithms for fluid computation based on the corrections that emerge from the compression of the algebraic structure into finite dimensions. These may play a role in proving long term existence of the classical ODE if that long term existence is true. If it is not true, this may also be revealed in these models which work at every scale. In either case we will have natural algorithms for computations of real fluids at every scale which of course do have long time existence and which are potentially observable at extremely small scales.

\end{document}